\documentclass[10pt,reqno,a4paper]{article}
\usepackage[english]{babel}
\usepackage{amsmath}
\usepackage{amsfonts}
\usepackage{amsthm}
\usepackage{amssymb}
\usepackage{enumerate}

\numberwithin{equation}{section}
\theoremstyle{plain}
\newtheorem{theo}{Theorem}[section]
\newtheorem{lem}[theo]{Lemma}
\newtheorem{prop}[theo]{Proposition}
\newtheorem{definition}[theo]{Definition}
\theoremstyle{remark}
\newtheorem{rem}[theo]{Remark}
\newcommand{\C}{\mathbb{C}}

\newcommand{\R}{\mathbb{R}}

\newcommand{\M}{\mathbb{M}}
\newcommand{\divrg}{\textrm{div}\,}

\begin{document}

\title{Global stability for an inverse problem in soil-structure
interaction
\thanks{The first and the third authors are supported by
FRA2014 `Problemi inversi per PDE, unicit\`a, stabilit\`a, algoritmi', Universit\`a degli Studi di Trieste,
the second author has been partially supported
by the Carlos III University of Madrid-Banco de Santander Chairs of Excellence Programme
for the 2013-2014 Academic Year, the third and the fourth author are partially supported by
GNAMPA of Istituto Nazionale di Alta Matematica.}
}
\author{Giovanni Alessandrini\thanks{
Universit\`{a} degli Studi di Trieste, Italy
(alessang@units.it)}\and Antonino Morassi\thanks{ Universit\`{a}
degli Studi di Udine, Italy (antonino.morassi@uniud.it)}\and Edi
Rosset\thanks{ Universit\`{a} degli Studi di Trieste, Italy
(rossedi@units.it)}\and Sergio Vessella\thanks{ Universit\`{a}
degli Studi di Firenze, Italy (sergio.vessella@unifi.it)}}
\date{}
\maketitle
\begin{abstract}

We consider the inverse problem of determining the Winkler
subgrade reaction coefficient of a slab foundation modelled as a
thin elastic plate clamped at the boundary. The plate is loaded by
a concentrated force and its transversal deflection is measured at
the interior points. We prove a global H\"{o}lder stability
estimate under (mild) regularity assumptions on
the unknown coefficient.

\end{abstract}



\section{Introduction\label{Intro}}

The soil-structure interaction is an important issue in structural
building design. The determination of the contact actions
exchanged between foundation and soil is commonly approached by
using simplified models of interaction. Among these, the model
introduced by Winkler in the second half of the nineteenth century is
one of the most popular in engineering applications \cite{Win}. In
Winkler's model, the foundation rests on a bed of linearly elastic
springs of stiffness $k$, $k \geq 0$, acting along the vertical
direction only. The springs are independent of each other, that
is, the deflection of every spring is not influenced by the other
adjacent springs. The accuracy of this model of interaction
depends strongly on the values assigned to the subgrade reaction
coefficient $k$. Ranges of average values of $k$ are available in
literature {}from extensive series of in-situ experiments
performed on various soil types, \cite{CG}, but these values are
quite scattered and, in addition, they may vary significantly
{}from point to point in the case of large foundations. Estimate
of the coefficient $k$ becomes even more difficult for existing
buildings, since the soil on which the foundation is resting is
not directly accessible for experiments. For the reasons stated
above, the development of a method for the determination of $k$ is
an inverse problem of current interest in practice.

In this paper we consider the stability issue in determining the
Winkler's subgrade coefficient of a slab foundation {}from the
measurement of the deflection induced at interior points by a
given load condition. The mechanical model is as follows. The slab
foundation is described as a thin elastic plate with uniform
thickness $h$ and middle surface coinciding with a bounded
Lipschitz domain $\Omega \subset \R^2$. The plate is assumed to be clamped at the boundary
$\partial \Omega$, a condition that occurs when the slab
foundation is anchored to sufficiently rigid vertical walls. A
concentrated force of intensity $f$ is supposed to act at an
internal point $P_0 \in \Omega$. This load condition has the merit
of being easy to implement in practice. According to the Winkler
model of soil and working in the framework of the Kirchhoff-Love
theory of plates, the transversal displacement $w$ of the plate
satisfies the fourth order Dirichlet boundary value problem
\begin{center}
\( {\displaystyle \left\{
\begin{array}{lr}
        \divrg (\divrg ( \frac{h^3}{12}\C \nabla^2 w)) + kw = f \delta(P_0), & \mathrm{in}\ \Omega,
    \vspace{0.25em}\\
    w =0, & \mathrm{on}\ \partial \Omega,
        \vspace{0.25em}\\
    \frac{\partial w}{\partial n} =0, & \mathrm{on}\ \partial \Omega,
          \vspace{0.25em}\\
\end{array}
\right. } \) \vskip -6.0em
\begin{eqnarray}
& & \label{eq-3I-1}\\
& & \label{eq-3I-2}\\
& & \label{eq-3I-3}
\end{eqnarray}
\end{center}
where $\C$ is the elasticity tensor of the material and $n$ is the
unit outer normal to $\partial \Omega$. Given the concentrated
force $f\delta(P_0)$ and the coefficient $k$, $k \in
L^\infty(\Omega)$, for a strongly convex tensor $\C \in L^\infty
(\Omega)$ the problem \eqref{eq-3I-1}--\eqref{eq-3I-3} admits a
unique solution $w \in H^2_0(\Omega)$.

The inverse problem in which we are interested in consists in
studying the stability of the determination of the unknown
subgrade coefficient $k$ in \eqref{eq-3I-1}--\eqref{eq-3I-3}
{}from a single measurement of $w$ inside $\Omega$. It should be
noted that the measurement of the transversal deflection at
interior points of the plate can be easily carried out by means of
no-contact techniques based on radar methodology (\cite {Be}).

Our main result states that, for $\C \in W^{2,\infty}(\Omega)\cap H^{2+s}(\Omega)$,
for some $0<s<1$,
satisfying a suitable structural condition (see \eqref{eq-7-1}), if $w_i \in
H^2_0(\Omega)$ is the solution to \eqref{eq-3I-1}--\eqref{eq-3I-3}
for Winkler coefficient $k=k_i \in L^\infty(\Omega) \cap
H^s(\Omega)$, $i=1,2$, and if, for a
given $\epsilon>0$,
\begin{equation}
    \label{4I-1}
    \|w_1 - w_2\|_{L^2(\Omega)} \leq \epsilon f,
\end{equation}
then, for every $\sigma > 0$, we have
\begin{equation}
    \label{4I-2}
    \|k_1 - k_2\|_{L^2(\Omega_{\sigma})} \leq C\epsilon^\beta,
\end{equation}
where $\Omega_{\sigma}= \{x \in \Omega\ | \ dist(x,\partial
\Omega) > \sigma\}$ and the constants $C>0$, $\beta \in
(0,1)$ only depend on the a priori data and on $\sigma$.

It should be noted that one difficulty of the problem comes
{}from the fact that the displacement $w$ may change sign
and vanish somewhere inside $\Omega$. See for instance the examples
in \cite{Ga}, \cite {KKM}, \cite{Sh-Te}. Therefore it is necessary to keep under control
the possible vanishing rate of $w$.
Thus, the key ingredients of the proof are quantitative versions of the unique
continuation principle for the solutions to the equation
$\divrg (\divrg ( \frac{h^3}{12}\C \nabla^2 w)) + kw = 0$, precisely an estimate
of continuation {}from an open subset to all of the domain (Propositions \ref{prop-8-1})
and the $A_p$ property (Proposition \ref{prop-9-1}). Another useful tool is
a pointwise lower bound
in a neighborhood of the point $P_0$ where the force is acting 
for solutions to \eqref{eq-3I-1} (Lemma \ref{lem-8a-1}).

Let us  mention that this method, essentially based on quantitative
estimates of unique continuation, has some similarities, although with
a different underlying equation and with a different kind of data,
with the one used in \cite{Al},
for another inverse problem with interior measurements arising in
hybrid imaging. 

The paper is organized as follows. Section \ref{Direct} contains the notation, the formulation of the direct problem and a regularity result in fractional Sobolev spaces (Proposition \ref{prop-5a-1}). Section
\ref{Inverse} is devoted to the formulation and analysis of the inverse problem.

\section{The direct problem\label{Direct}}

\subsection{Notation\label{Notation}}

We shall denote by $B_r(P)$ the open disc in $\R^2$ of radius $r$ and
center $P$.

For any $U \subset \R^2$ and for any $r>0$, we denote
\begin{equation}
  \label{eq:2.int_env}
  U_{r}=\{x \in U \ |\  \textrm{dist}(x,\partial U)>r
  \}.
\end{equation}
\begin{definition}
  \label{def:2.1} (${C}^{k,\alpha}$ regularity)
Let $\Omega$ be a bounded domain in ${\R}^{2}$. Given $k,\alpha$,
with $k=0,1,2,...$, $0<\alpha\leq 1$, we say that a portion $S$ of
$\partial \Omega$ is of \textit{class ${C}^{k,\alpha}$ with
constants $\rho_{0}$, $M_{0}>0$}, if, for any $P \in S$, there
exists a rigid transformation of coordinates under which we have
$P=O$ and
\begin{equation*}
  \Omega \cap B_{\rho_{0}}(O)=\{x=(x_1,x_2) \in B_{\rho_{0}}(O)\quad | \quad
x_{2}>\psi(x_1)
  \},
\end{equation*}
where $\psi$ is a ${C}^{k,\alpha}$ function defined in
$I_{\rho_0}=(-\rho_{0},\rho_{0})$
satisfying
\begin{equation*}
\psi(0)=0,
\end{equation*}
\begin{equation*}
\psi' (0)=0, \quad \hbox {when } k \geq 1,
\end{equation*}
\begin{equation*}
\|\psi\|_{{C}^{k,\alpha}(I_{\rho_0})} \leq M_{0}\rho_{0}.
\end{equation*}

\medskip
\noindent When $k=0$, $\alpha=1$, we also say that $S$ is of
\textit{Lipschitz class with constants $\rho_{0}$, $M_{0}$}.
\end{definition}

  We use the convention to normalize all norms in such a way that their
  terms are dimensionally homogeneous with the argument of the norm and coincide with the
  standard definition when the dimensional parameter equals one.
  For instance, the norm appearing above is meant as follows
\begin{equation*}
  \|\psi\|_{{C}^{k,\alpha}(I_{\rho_0})} =
  \sum_{i=0}^k \rho_0^i
  \|\psi^{(i)}\|_{{L}^{\infty}(I_{\rho_0})}+
  \rho_0^{k+\alpha}|\psi^{(k)}|_{\alpha, I_{\rho_0}},
\end{equation*}
where
\begin{equation*}
|\psi^{(k)}|_{\alpha, I_{\rho_0}}= \sup_
{\overset{\scriptstyle x_1, \ y_1\in I_{\rho_0}}{\scriptstyle
x_1\neq y_1}} \frac{|\psi^{(k)}(x_1)-\psi^{(k)}(y_1)|}
{|x_1-y_1|^\alpha}
\end{equation*}
and $\psi^{(i)}$ denotes the $i$-order
derivative of $\psi$.

Similarly, given a function $u:\Omega\mapsto \R$, where $\partial
\Omega$ satisfies Definition \ref{def:2.1}, and denoting by
$\nabla^i u$ the vector which components are the derivatives of
order $i$ of the function $u$, we denote
\begin{equation*}
\|u\|_{L^2(\Omega)}=\rho_0^{-1}\left(\int_\Omega u^2\right)
^{\frac{1}{2}},
\end{equation*}
\begin{equation*}
\|u\|_{H^k(\Omega)}= \rho_0^{-1} \left ( \sum_{i=0}^{k}
\rho_0^{2i} \int_\Omega |\nabla^i u|^2 \right )^{ \frac{1}{2} }, \quad k=0,1,2,...
\end{equation*}

Moreover, for $k=0,1,2,...$, and $s \in (0,1)$, we denote
\begin{equation*}
\|u\|_{H^{k+s}(\Omega)}= \|u\|_{H^k(\Omega)} + \rho_0^{s-1}[u]_{s},
\end{equation*}
where the semi-norm $[\ \!\cdot\ \!]_s$ is given by
\begin{equation}
   \label{eq:2.notation_0}
[u]_{s} = \left ( \int_\Omega \int_\Omega \frac{
|u(x)-u(y)|^2 }{ |x-y|^{2+2s} }\ dx \ dy \right )^{ \frac{1}{2}  }.
\end{equation}

For every $2 \times 2$ matrices $A$, $B$ and every $\mathbb L
\in \mathcal{L}(\M^2 \times \M^2)$, we use the following notation:
\begin{equation}
  \label{eq:2.notation_1}
  ({\mathbb{L}}A)_{ij} = L_{ijkl}A_{kl},
\end{equation}
\begin{equation}
  \label{eq:2.notation_2}
  A \cdot B = A_{ij}B_{ij},
\end{equation}
\begin{equation}
  \label{eq:2.notation_3}
  |A|= (A \cdot A)^{\frac {1} {2}}.
\end{equation}
Finally, we denote by $A^T$ the transpose of the matrix $A$.

\subsection{Formulation of the direct problem\label{Formulation}}

Let $\Omega \subset \R^2$ be a bounded domain whose boundary is of Lipschitz class with constants
$\rho_0$, $M_0$ and assume that
\begin{equation}
    \label{eq-2-1}
    |\Omega| \leq M_1 \rho_0^2.
\end{equation}
We consider a thin plate $\Omega \times \left [ -\frac{h}{2},
\frac{h}{2} \right ]$ with middle surface represented by $\Omega$ and whose thickness $h$ is much smaller than the characteristic dimension of $\Omega$, that is $h << \rho_0$. 
The plate is made by linearly elastic material with elasticity
tensor $\C(\cdot) \in L^\infty(\Omega, {\mathcal{L}}(\M^2,\M^2))$ with cartesian
components $C_{\alpha \beta \gamma \delta}$
satisfying the symmetry conditions
\begin{equation}
    \label{eq-2-2}
    \C A = (\C A)^T,
\end{equation}
\begin{equation}
    \label{eq-2-3}
    \C A \cdot B = A \cdot \C B,
\end{equation}
for every $2 \times 2$ matrices $A$, $B$, and the strong convexity condition
\begin{equation}
    \label{eq-2-4}
    \xi_0 |A|^2 \leq \C A \cdot A \leq \xi_1 |A|^2,
\end{equation}
for every $2 \times 2$ symmetric matrix $A$, where $\xi_0$,
$\xi_1$ are positive constants.

The plate is resting on a Winkler soil with subgrade reaction
coefficient
\begin{equation}
    \label{eq-3-0}
    k \in L^\infty(\Omega),\quad  k \geq 0 \ \
    \hbox{a.e. in } \Omega.
\end{equation}
The boundary $\partial \Omega$ is clamped and we assume that a
concentrated force is acting at a point $P_0 \in \Omega$ along a
direction orthogonal to the middle surface $\Omega$. According to the
Kirchhoff-Love theory of thin plates subject to infinitesimal
deformation, the statical equilibrium of the plate is described by
the following Dirichlet boundary value problem
\begin{center}
\( {\displaystyle \left\{
\begin{array}{lr}
        \divrg (\divrg ( \mathbb P \nabla^2 w)) + kw = f \frac{\delta(P_0)}{\rho_0^2}, & \mathrm{in}\ \Omega,
    \vspace{0.25em}\\
    w =0, & \mathrm{on}\ \partial \Omega,
        \vspace{0.25em}\\
    \frac{\partial w}{\partial n} =0, & \mathrm{on}\ \partial \Omega,
          \vspace{0.25em}\\
\end{array}
\right. } \) \vskip -6.0em
\begin{eqnarray}
& & \label{eq-3-1}\\
& & \label{eq-3-2}\\
& & \label{eq-3-3}
\end{eqnarray}
\end{center}
where the plate tensor $\mathbb P$ is given by
\begin{equation}
    \label{eq-3-4}
    \mathbb P = \frac{h^3}{12} \C;
\end{equation}
the subgrade reaction coefficient $k$ satisfies
\begin{equation}
    \label{eq-3-5}
    0 \leq k(x) \leq \frac{\overline{k}}{\rho_0^4}, \quad
    \hbox{a.e. in } \ \Omega,
\end{equation}
for some positive constant $\overline{k}$; the concentrated force
is positive, i.e.,
\begin{equation}
    \label{eq-3-6}
    f\in \R,\ f >0;
\end{equation}
$w=w(x)$ is the transversal displacement at the point $x    \in
\Omega$ and $n$ is the unit outer normal to $\partial \Omega$.

We notice that the presence in
\eqref{eq-3-1} and \eqref{eq-3-5} of the parameter $\rho_0$ (which has
the dimension of a length) allows for a scaling-invariant
formulation of the plate equation.

\begin{prop}
\label{prop-4-1}
Under the above assumptions, there exists a unique weak solution
$w \in H^2_0(\Omega)$ to \eqref{eq-3-1}--\eqref{eq-3-3}, which
satisfies
\begin{equation}
    \label{eq-4-1}
    \|w\|_{H^2(\Omega)} \leq Cf,
\end{equation}
where the constant $C>0$ only depends on $h$, $M_0$, $M_1$,
$\xi_0$.
\end{prop}
\begin{proof}
The weak formulation of the problem \eqref{eq-3-1}--\eqref{eq-3-3}
consists in finding $w \in H_0^2(\Omega)$ such that
\begin{equation}
    \label{eq-4-2}
    \int_\Omega \mathbb P \nabla^2 w \cdot \nabla^2 v +
    \int_\Omega kwv = \frac{f}{\rho_0^2 } v(P_0), \quad \hbox{for
    every } v \in H_0^2 (\Omega).
\end{equation}
Let us notice that
\begin{equation}
    \label{eq-4-3}
    H_0^2(\Omega) \subset C^{0,\alpha}(\overline{\Omega}), \quad
    \hbox{for every } \alpha < 1
\end{equation}
and, therefore, the linear functional
\begin{equation*}
F:H^2_0(\Omega)\rightarrow\R
\end{equation*}
\begin{equation*}
F(v)= \frac{f}{\rho_0^2 } v(P_0)
\end{equation*}
is bounded and the symmetric bilinear form
$B(u,v)= \int_\Omega \mathbb P \nabla^2 w \cdot \nabla^2 v +
\int_\Omega kwv $  is bounded and coercive on $H^2_0(\Omega)\times H^2_0(\Omega)$.
By Riesz representation theorem a solution to \eqref{eq-4-2} exists and is unique.
By choosing $v=w$ in \eqref{eq-4-2}, by \eqref{eq-2-4} and using
Poincar\'{e} inequality, we have
\begin{equation}
    \label{eq-5-1}
    \frac{fw(P_0)}{\rho_0^2} \geq \int_{\Omega} \mathbb P \nabla^2
    w \cdot \nabla^2 w \geq \frac{C}{\rho_0^2}
    \|w\|_{H^2(\Omega)}^2,
\end{equation}
where the constant $C>0$ only depends on $h$, $M_0$, $M_1$, $\xi_0$. By
\eqref{eq-5-1} and the embedding \eqref{eq-4-3}, the thesis
follows.
\end{proof}

In the analysis of the inverse problem, we shall need the
following regularity result when the coefficients of the plate operator belong to a
fractional Sobolev space.

\begin{prop}[$H^s$-regularity]
\label{prop-5a-1}
Let $\Omega$ be a bounded domain in $\R^2$ with boundary of Lipschitz
class with constants $\rho_0$, $M_0$, satisfying \eqref{eq-2-1}. Given $g \in H^s(\Omega)$, let $w \in H^2(\Omega)$ be a solution
to
\begin{equation}
    \label{eq-5a-1}
    \divrg (\divrg ( \mathbb P \nabla^2 w)) = g, \quad \hbox{in } \Omega,
\end{equation}
where $\mathbb P$ is given by \eqref{eq-3-4}, with $\C$ satisfying
\eqref{eq-2-2}--\eqref{eq-2-4} and, for some $s \in (0,1)$,
\begin{equation}
    \label{eq-5a-2}
    \|\C \|_{W^{2,\infty}(\Omega)} \leq M_2,
\end{equation}
\begin{equation}
    \label{eq-5a-3}
    \|\C \|_{H^{2+s}(\Omega)} \leq M_3.
\end{equation}
Then, for every $\sigma > 0$, we have
\begin{equation}
    \label{eq-5a-4}
    \|w \|_{H^{4+s}(\Omega_{\sigma \rho_0})} \leq C\left (
    \|w\|_{H^2(\Omega_{\frac{\sigma}{2}\rho_0})} + \rho_0^4 \|g\|_{H^{s}(\Omega)}
    \right ),
\end{equation}
where the constant $C>0$ only depends on $h$, $M_0$, $M_1$, $M_2$,
$M_3$,
$\xi_0$, $s$, $\sigma$.
\end{prop}
\begin{proof}
When $\mathbb{C\in }C^{\infty }\left( \Omega \right) $ the estimate \eqref{eq-5a-4}
is a form of the well-known classical Garding's inequality. Under the less
restrictive condition \eqref{eq-5a-2}, \eqref{eq-5a-3}, the proof of \eqref{eq-5a-4} can be carried out
following the same path traced in the classical case (\cite{Ag}, \cite{Fo}) taking
care to control the lower order terms by means of $M_{2}$ and $M_{3}$. We omit the
details.
\end{proof}

\section{The inverse problem\label{Inverse}}

In order to derive our stability result for the inverse problem we
need further a priori information.

Concerning the point $P_0$ of the plate in which the concentrated
force is acting, we assume that
\begin{equation}
    \label{eq-6-1}
    dist(P_0, \partial \Omega) \geq d\rho_0,
\end{equation}
for some positive constant $d$. On the elasticity tensor $\C=\{C_{\alpha \beta \gamma \delta}\}$ we
further assume the stronger regularity \eqref{eq-5a-2}, \eqref{eq-5a-3} and,
moreover, we introduce a structural
condition. Precisely, denoting by $a_0=C_{1111}$,
$a_1=4C_{1112}$, $a_2=2C_{1122}+4C_{1212}$, $a_3=4C_{2212}$,
$a_4=C_{2222}$ and by $S(x)$ the following matrix
\begin{equation}
    \label{eq-6-3}
    S(x)=\left(\begin{array}{ccccccc}
a_0 & a_1 & a_2 & a_3 & a_4 & 0 & 0\\
0 & a_0 & a_1 & a_2 & a_3 & a_4 & 0\\
0 & 0 & a_0 & a_1 & a_2 & a_3 & a_4\\
4a_0 & 3a_1 & 2a_2 & a_3 & 0 & 0 & 0\\
0 & 4a_0 & 3a_1 & 2a_2 & a_3 & 0 & 0\\
0 & 0 & 4a_0 & 3a_1 & 2a_2 & a_3 & 0\\
0 & 0 & 0 & 4a_0 & 3a_1 & 2a_2 & a_3
\end{array}\right),
\end{equation}
we assume that
\begin{equation}
    \label{eq-7-1}
    \mathcal{D}(x)=0,\quad \hbox{for every } x \in \Omega,
\end{equation}
where
\begin{equation}
    \label{eq-7-2}
    \mathcal{D}(x)=\frac{1}{a_0}|\mathrm{det}S(x)|.
\end{equation}
Let us recall that condition \eqref{eq-7-1} includes the class of
orthotropic materials and, in particular, the isotropic Lam\'{e}
case, see \cite{Mo-Ro-Ve11}. Concerning the subgrade reaction
coefficient $k$, we require the additional regularity
\begin{equation}
    \label{eq-7-3}
    \rho_0^{s-1}[k]_{H^s(\Omega)} \leq
    \frac{\overline{k}}{\rho_0^4}.
\end{equation}
\begin{rem}
\label{rem-Hs}
Let us emphasize that the assumption $k\in H^s(\Omega)$ is not merely a mathematical
technicality, but it can be grounded on realistic mechanical considerations.
If, for instance, $k$ is piecewise constant and is represented as
\begin{equation}
    \label{eq-characteristic}
    k(x) = \sum_{j=1}^J k_j\chi_{E_j}(x), \quad \hbox{for every } x\in \R^2,
\end{equation}
where  $k_j\in \R$ and $E_1$,...,$E_J$ is a partition of $\Omega$ into disjoint subsets of finite perimeter
(in the sense of Caccioppoli, that is $\chi_{E_j} \in BV(\R^2)$ for every $j$), then
$k$ belongs to $H^s(\R^2)$ for every $s$, $0<s<\frac{1}{2}$.
In fact one has
\begin{equation*}
    [k]_{s}^2 \leq
    C_s\|k\|_{L^\infty}^{2s} \left(\int_{\R^2} |k|^2\right)^{1-2s}\left(\int_{\R^2} |Dk|\right)^{2s},
\end{equation*}
for every $k\in L^2(\R^2)\cap L^\infty(\R^2)\cap BV(\R^2)$.
Here $C_s$ only depends on $s\in(0,\frac{1}{2})$ and
$\int_{\R^2} |Dk|$ denotes the total variation of $k$.
For a proof see \cite[formula (2.15)]{Ma-Pa85} and also \cite[Remark 1.16]{Gi} for the convergence properties
of the mollifications of $BV$ functions.

In particular, if $k$ is of the form \eqref{eq-characteristic} and we assume
\begin{equation*}
    P(E_j)=\int_{\R^2} |D\chi_{E_j}|\leq \mathcal{P}\rho_0,\quad \hbox{for every } j=1,...,J,
\end{equation*}

for a given $\mathcal{P}>0$, then we obtain
\begin{equation*}
    [k]_{s}^2 \leq
    C_s \overline{k}^{2} M_1^{1-2s} (J\mathcal{P})^{2s}\rho_0^{-6-2s}.
\end{equation*}
\end{rem}

Hereinafter, we shall refer to $h$, $d$, $M_0$, $M_1$, $M_2$,
$M_3$, $\xi_0$, $\overline{k}$, $s$ as the \textit{a priori data}.

\begin{theo}
\label{theo-7-1}

Let $\Omega$ be a bounded domain in $\R^2$ with boundary of
Lipschitz class with constants $\rho_0$, $M_0$, satisfying
\eqref{eq-2-1}. Let $\mathbb P$ given by \eqref{eq-3-4}, with $\C
\in W^{2,\infty}(\Omega)\cap H^{2+s}(\Omega)$ satisfying
\eqref{eq-2-2}--\eqref{eq-2-4}, \eqref{eq-5a-2}, \eqref{eq-5a-3} for some $s\in (0,1)$, and
\eqref{eq-7-1}. Let $P_0 \in \Omega$
satisfying \eqref{eq-6-1}.

Given $f > 0$, let $w_i \in H_0^2(\Omega)$, $i=1,2$, be the solution
to
\begin{center}
\( {\displaystyle \left\{
\begin{array}{lr}
        \divrg (\divrg ( \mathbb P \nabla^2 w_i)) + k_i w_i = f \frac{\delta(P_0)}{\rho_0^2}, & \mathrm{in}\ \Omega,
    \vspace{0.25em}\\
    w_i =0, & \mathrm{on}\ \partial \Omega,
        \vspace{0.25em}\\
    \frac{\partial w_i}{\partial n} =0, & \mathrm{on}\ \partial
    \Omega,
          \vspace{0.25em}\\
\end{array}
\right. } \) \vskip -6.0em
\begin{eqnarray}
& & \label{eq-7-4}\\
& & \label{eq-7-5}\\
& & \label{eq-7-6}
\end{eqnarray}
\end{center}
for $k_i \in L^\infty(\Omega) \cap
H^s(\Omega)$ satisfying \eqref{eq-3-5} and \eqref{eq-7-3}.

If, for some $\epsilon>0$,
\begin{equation}
    \label{eq-8-1}
    \|w_1 - w_2\|_{L^2(\Omega)}\leq \epsilon f,
\end{equation}
then for every $\sigma > 0$ we have
\begin{equation}
    \label{eq-8-2}
    \|k_1 - k_2\|_{L^2(\Omega_{\sigma \rho_0})}\leq \frac{C}{\rho_0^4} \epsilon^\beta,
\end{equation}
where the constants $C>0$ and $\beta \in (0,1)$ only depend on the
a priori data and on $\sigma$.
\end{theo}

As is obvious, the above stability result also implies uniqueness. Indeed, by the following arguments it is easily seen that, under the above stated structural conditions on $\C$ 
\eqref{eq-6-3}--\eqref{eq-7-2}, uniqueness continues to hold by merely assuming $k\in L^\infty$ and $\C\in W^{2,\infty}$.

Let us premise to the proof of Theorem \ref{theo-7-1} some
auxiliary propositions concerning quantitative versions of the unique
continuation principle (Lemma \ref{lem-8a-1} and Propositions \ref{prop-8-1}
and \ref{prop-9-1} below).

\begin{lem}
\label{lem-8a-1}
Let $\Omega$ be a bounded domain with boundary $\partial \Omega$
of Lipschitz class with constants $\rho_0$, $M_0$, satisfying
\eqref{eq-2-1}. Let $P_0 \in \Omega$ satisfying \eqref{eq-6-1}.
Let $\mathbb P$ given by \eqref{eq-3-4}, with $\C$ satisfying
\eqref{eq-2-2}--\eqref{eq-2-4}, and let $k$ and $f$ satisfy
\eqref{eq-3-5}, \eqref{eq-3-6}, respectively. Let $w \in
H_0^2(\Omega)$ be the solution to \eqref{eq-3-1}--\eqref{eq-3-3}.
There exists $\overline{\sigma} >0$, only depending on $h$, $d$,
$M_0$, $M_1$, $\xi_0$, $\xi_1$, such that
\begin{equation}
    \label{eq-8a-0}
    w(x)\geq Cd^2f,\quad \forall x\in B_{2\overline{\sigma}\rho_0}(P_0),
\end{equation}
where $C>0$ only depends on $h$, $M_0$, $M_1$,
$\xi_0$ and $\xi_1$,
\begin{equation}
    \label{eq-8a-1}
    \int_{B_{2\sigma \rho_0}(P_0)\setminus B_{\sigma
    \rho_0}(P_0)} w^2 \geq C \sigma^2 d^2 \rho_0^2
    \|w\|_{H^2(\Omega)}^2,\quad \hbox{for every }\sigma,\ 0 <
\sigma \leq \overline{\sigma}
\end{equation}
where $C>0$ only depends on $h$, $M_0$, $M_1$,
$\xi_0$, $\xi_1$, $\overline{k}$.
\end{lem}
\begin{proof}
By \eqref{eq-2-4} and \eqref{eq-4-2}, we have that for every $v
\in H_0^2(\Omega)$
\begin{equation}
    \label{eq-8a-2}
    f|v(P_0)| \leq C \|v\|_{H^2(\Omega)}\|w\|_{H^2(\Omega)},
\end{equation}
so that
\begin{equation}
    \label{eq-8b-1}
    \|\delta(P_0)\|_{H^{-2}(\Omega)} = \sup_
    {\overset{\scriptstyle v\in H_0^2(\Omega)}{\scriptstyle
     \|v\|_{H^2(\Omega)}=1}}
     |v(P_0)| \leq \frac{C}{f}
    \|w\|_{H^2(\Omega)},
\end{equation}
where $C>0$ only depends on $h$, $\xi_1$, $\overline{k}$.
Since $B_{d\rho_0}(P_0) \subset \Omega$ by \eqref{eq-6-1}, we have
\begin{equation}
    \label{eq-8b-2}
    \|\delta(P_0)\|_{H^{-2}(\Omega)} \geq
    \|\delta(P_0)\|_{H^{-2}(B_{d\rho_0}(P_0))} \geq Cd,
\end{equation}
where $C>0$ is an absolute constant. {}From \eqref{eq-8b-1},
\eqref{eq-8b-2} it follows that
\begin{equation}
    \label{eq-8b-3}
    \|w\|_{H^2(\Omega)} \geq Cdf,
\end{equation}
where $C>0$ only depends on $h$, $\xi_1$, $\overline{k}$.
By \eqref{eq-2-4}, \eqref{eq-4-2} and Poincar\'{e} inequality, we have
\begin{equation}
    \label{eq-8b-4}
    w(P_0) \geq \frac{C}{f}\|w\|_{H^2(\Omega)}^2,
\end{equation}
where $C>0$ only depends on $h$, $M_0$, $M_1$, $\xi_0$. By
\eqref{eq-8b-3}, \eqref{eq-8b-4} and by the embedding inequality
\eqref{eq-4-3}, we have
\begin{equation}
    \label{eq-8b-5}
    w(P_0) \geq Cd \|w\|_{H^2(\Omega)},
\end{equation}
\begin{equation}
    \label{eq-8b-6}
    w(P_0) \geq c_0 d \|w\|_{C^{0,\alpha}(\overline{\Omega})},
\end{equation}
where $C>0$ and $c_0 >0$ only depend on $h$, $M_0$, $M_1$,
$\xi_0$, $\xi_1$, $\overline{k}$. Let
\begin{equation}
    \label{eq-8c-1}
    \overline{\sigma} = \min \left ( \frac{d}{4}, \frac{1}{2}
    \left ( \frac{c_0 d}{2}\right )^{\frac{1}{\alpha}} \right ).
\end{equation}
Let us notice that, by this choice of $\overline{\sigma}$,
$dist(P_0, \partial \Omega) \geq 4 \overline{\sigma} \rho_0$ and,
recalling \eqref{eq-8b-6}, we have
\begin{multline}
    \label{eq-8c-2}
    w(x) \geq w(P_0) - |w(x)-w(P_0)| \geq w(P_0) -
    (2\overline{\sigma})^\alpha\|w\|_{C^{0,\alpha}(\overline{\Omega})} \geq
    \frac{w(P_0)}{2}, \\ \hbox{for every } x \in
    B_{2\overline{\sigma}\rho_0}(P_0).
\end{multline}
Choosing $\alpha = \frac{1}{2}$, \eqref{eq-8a-0} follows {}from \eqref{eq-8b-3}, \eqref{eq-8b-4}
and \eqref{eq-8c-2} whereas \eqref{eq-8a-1} follows, restricting to the annulus
$B_{2\sigma \rho_0}(P_0)\setminus B_{\sigma
    \rho_0}(P_0)$, which is contained in $B_{2\overline{\sigma}\rho_0}(P_0)$ for $\sigma\leq \overline{\sigma}$, {}from \eqref{eq-8b-5} and \eqref{eq-8c-2}.
\end{proof}
\begin{prop}[Lipschitz propagation of smallness]
\label{prop-8-1}
Let $U$ be a bounded Lipschitz domain of $\R^2$ with constants
$\rho_0$, $M_0$ and satisfying $|U| \leq M_1 \rho_0^2$. Let $w \in
H^2(U)$ be a solution to
\begin{equation}
    \label{eq-8bis-1}
 \divrg (\divrg ( \mathbb P \nabla^2 w)) + k w = 0, \quad \hbox{in } U,
\end{equation}
where $\mathbb P$, defined in \eqref{eq-3-4}, satisfies
\eqref{eq-2-2}--\eqref{eq-2-4} and \eqref{eq-5a-2} in $U$, and $k$
satisfies \eqref{eq-3-5} in $U$. Assume
\begin{equation*}
\frac{\|w\|_{H^{ \frac{1}{2} } (U)}}{\|w\|_{L^2(U)}}\leq N.
\end{equation*}
There exists a constant $c_1 >
1$, only depending on $h$, $M_2$, $\xi_0$ and $\overline{k}$, such
that, for every $\tau >0$ and for every $x \in U_{c_1 \tau
\rho_0}$, we have
\begin{equation}
    \label{eq-9-1}
    \int_{B_{\tau \rho_0}(x)} w^2 \geq c_\tau \int_U w^2,
\end{equation}
where $c_\tau >0$ only depends on  $h$, $M_0$, $M_1$, $M_2$,
$\xi_0$, $\overline{k}$, $\tau$ and on $N$.
\end{prop}
The proof of the above proposition is based on the three spheres
inequality obtained in \cite{Lin-Nag-Wan_2011}.
\begin{prop}[$A_p$ property]
\label{prop-9-1}
In the same hypotheses of Proposition \ref{prop-8-1}, there exists
a constant $c_2 > 1$, only depending on $h$, $M_0$,
$M_1$, $M_2$, $\xi_0$, $\overline{k}$, such that, for every $\tau
> 0$ and for every $x \in U_{c_2\tau \rho_0}$, we have
\begin{equation}
    \label{eq-9-2}
    \left (
    \frac{1}{|B_{\tau \rho_0}(x)|}
    \int_{B_{\tau \rho_0}(x)} |w|^2
    \right )
    \left (
    \frac{1}{|B_{\tau \rho_0}(x)|}
    \int_{B_{\tau \rho_0}(x)} |w|^{ - \frac{2}{p-1}  }
    \right )^{  p-1  } \leq B,
\end{equation}
where $B>0$ and $p > 1$ only depend on $h$, $M_0$, $M_1$, $M_2$,
$\xi_0$, $\overline{k}$, $\tau$ and on $N$.
\end{prop}
The proof of the above proposition follows {}from the doubling inequality
obtained in \cite{dC-L-M-R-V-W}, by applying the arguments in \cite{G-L}.

\begin{proof}[Proof of Theorem \ref{theo-7-1}]
If $\epsilon\geq 1$, then the proof of \eqref{eq-8-2} is trivial in view of
\eqref{eq-3-5}. Therefore we restrict the analysis to the case $0<\epsilon<1$.

The difference
\begin{equation}
    \label{eq-10-1}
    w=w_1 - w_2
\end{equation}
of the solutions to \eqref{eq-7-4}--\eqref{eq-7-6} for $i=1,2$
satisfies the boundary value problem
\begin{center}
\( {\displaystyle \left\{
\begin{array}{lr}
        \divrg (\divrg ( \mathbb P \nabla^2 w)) + k_2 w = (k_2-k_1) w_1, & \mathrm{in}\ \Omega,
    \vspace{0.25em}\\
    w =0, & \mathrm{on}\ \partial \Omega,
        \vspace{0.25em}\\
    \frac{\partial w}{\partial n} =0, & \mathrm{on}\ \partial
    \Omega.
          \vspace{0.25em}\\
\end{array}
\right. } \) \vskip -6.0em
\begin{eqnarray}
& & \label{eq-10-2}\\
& & \label{eq-10-3}\\
& & \label{eq-10-4}
\end{eqnarray}
\end{center}
Obviously, it is not restrictive to assume that
$\sigma \leq \overline{\sigma}$, where
$\overline{\sigma}$ has been defined in \eqref{eq-8c-1} with $
\alpha = \frac{1}{2}$. We have
\begin{equation}
    \label{eq-10-5}
    \int_{\Omega_{\sigma \rho_0}} (k_2 -k_1)^2 w_1^2 \leq 2(I_1
    +I_2),
\end{equation}
where
\begin{equation}
    \label{eq-10-6}
    I_1 = \int_{\Omega_{\sigma \rho_0}} k_2^2 w^2 ,
\end{equation}
\begin{equation}
    \label{eq-10-7}
    I_2 = \int_{\Omega_{\sigma \rho_0}} \left ( \divrg (\divrg ( \mathbb P \nabla^2 w)) \right
    )^2.
\end{equation}
By \eqref{eq-3-5} and \eqref{eq-8-1}, we have
\begin{equation}
    \label{eq-10-8}
    I_1 \leq \frac{\overline{k}^2}{\rho_0^6} \epsilon^2.
\end{equation}
By
\eqref{eq-5a-2}, we have
\begin{equation}
    \label{eq-11-1}
    I_2 \leq C\frac{h^3M_2^2}{12\rho_0^6} \|w\|_{H^4(\Omega_{\sigma \rho_0})}^2.
\end{equation}
with $C>0$ an absolute constant.
Let $g= (k_2-k_1)w_1 -k_2w$. Note that, by \eqref{eq-3-5}, \eqref{eq-4-1}, \eqref{eq-4-3},
and \eqref{eq-7-3},

\begin{equation}
    \label{eq-11-2}
    \|g\|_{H^s(\Omega)} \leq C \frac{\overline{k}f}{\rho_0^4},
\end{equation}
where $C>0$ only depends on $h$, $M_0$, $M_1$, $\xi_0$.
By applying Proposition
\ref{prop-5a-1}, we
have
\begin{equation}
    \label{eq-11-3}
    \|w\|_{H^{4+s}(\Omega_{\sigma \rho_0})} \leq C f
    \overline{k},
\end{equation}
with $C>0$ only depending on $h$, $M_0$, $M_1$, $M_2$, $M_3$,
$\xi_0$, $s$, $\sigma$. {}From the well-known interpolation
inequality
\begin{equation}
    \label{eq-11-4}
    \|w\|_{H^{4}(\Omega_{\sigma \rho_0})} \leq C \|w\|_{H^{4+s}(\Omega_{\sigma
    \rho_0})}^{  \frac{4}{4+s}  }
    \|w\|_{L^{2}(\Omega_{\sigma
    \rho_0})}^{  \frac{s}{4+s}  },
\end{equation}
and recalling \eqref{eq-11-3} and \eqref{eq-8-1}, we obtain
\begin{equation}
    \label{eq-11-5}
    \|w\|_{H^{4}(\Omega_{\sigma \rho_0})} \leq C f
    \epsilon^{  \frac{s}{4+s}  },
\end{equation}
where $C>0$ only depends on $h$, $M_0$, $M_1$, $M_2$, $M_3$,
$\xi_0$, $\overline{k}$, $s$, $\sigma$. {}From
\eqref{eq-10-5}, \eqref{eq-10-8}, \eqref{eq-11-1},
\eqref{eq-11-5}, it follows that
\begin{equation}
    \label{eq-12-1}
    \int_{\Omega_{\sigma \rho_0}} (k_2 -k_1)^2 w_1^2 \leq
    \frac{C}{\rho_0^6}f^2\epsilon^{ \frac{2s}{4+s}  },
\end{equation}
where $C>0$ only depends on $h$, $M_0$, $M_1$, $M_2$, $M_3$,
$\xi_0$, $\overline{k}$, $s$, $\sigma$.

Let us first estimate $|k_2-k_1|$ in a disc centered at $P_0$.
Notice that, by the choice of $\overline{\sigma}$, $\Omega_{\sigma\rho_0}\supset B_{2\overline{\sigma}\rho_0}(P_0)$, for every $\sigma\leq\overline{\sigma}$.
By applying \eqref{eq-8a-0} for $w=w_1$, and by \eqref{eq-12-1} with $\sigma=\overline{\sigma}$, we obtain
\begin{equation}
    \label{eq-12-2}
    \int_{B_{2\overline{\sigma} \rho_0}(P_0)} (k_2 -k_1)^2 \leq
    \frac{C}{\rho_0^6 d^4}\epsilon^{ \frac{2s}{4+s}  },
\end{equation}
where $C>0$ only depends on $h$, $M_0$, $M_1$, $M_2$, $M_3$,
$\xi_0$, $\overline{k}$, $s$ and $d$.

Now, let us control $|k_2-k_1|$ in
\begin{equation}
    \label{eq-13-1}
    \widetilde{\Omega}_{\sigma \rho_0} = \Omega_{\sigma \rho_0}
    \setminus B_{2\overline{\sigma} \rho_0}.
\end{equation}
This estimate is more involved and requires arguments of unique
continuation, precisely the $A_p$-property and the Lipschitz
propagation of smallness.

By applying H\"{o}lder inequality and \eqref{eq-12-1}, we can
write, for every $p > 1$,
\begin{multline}
    \label{eq-13-2}
    \int_{\widetilde{\Omega}_{\sigma \rho_0}}(k_2-k_1)^2 =
    \int_{\widetilde{\Omega}_{\sigma \rho_0}}|w_1|^{ \frac{2}{p}
    }(k_2-k_1)^2|w_1|^{ -\frac{2}{p}  } \leq \\
    \leq
    \left (
    \int_{\widetilde{\Omega}_{\sigma \rho_0}}
    (k_2-k_1)^2 w_1^2
    \right )^{  \frac{1}{p} }
    \left (
    \int_{\widetilde{\Omega}_{\sigma \rho_0}}
    (k_2-k_1)^2 |w_1|^{ - \frac{2}{p-1}  }
    \right )^{  \frac{p-1}{p} } \leq \\
    \leq
    \frac{C}{\rho_0^{  \frac{6}{p}  }} f^{ \frac{2}{p}   }
    \epsilon^{ \frac{2s}{p(4+s)}   }
    \left (
    \int_{\widetilde{\Omega}_{\sigma \rho_0}}
    (k_2-k_1)^2 |w_1|^{ - \frac{2}{p-1}  }
    \right )^{  \frac{p-1}{p} },
\end{multline}
where $C>0$ only depends on $h$, $M_0$, $M_1$, $M_2$, $M_3$,
$\xi_0$, $\overline{k}$, $s$, $\sigma$.

Let us cover $\widetilde{\Omega}_{\sigma \rho_0}$ with internally
non overlapping closed squares $Q_l(x_j)$ with center $x_j$ and side $l
= \frac{\sqrt{2}}{2\max\{2,c_1,c_2\}}\sigma \rho_0$, $j=1,...,J$,
where $c_1$ and $c_2$ have been introduced in Proposition
\ref{prop-8-1} and in Proposition \ref{prop-9-1}, respectively. By
the choice of $l$, denoting $r= \frac{\sqrt{2}}{2}l$,
\begin{equation}
    \label{eq-14-1}
    \widetilde{\Omega}_{\sigma \rho_0} \subset \bigcup_{j=1}^J Q_l(x_j)
    \subset \bigcup_{j=1}^J B_r(x_j) \subset
    \Omega_{\frac{\sigma}{2}
    \rho_0}\setminus B_{\overline{\sigma}\rho_0}(P_0),
\end{equation}
so that
\begin{equation}
    \label{eq-14-2}
    \int_{\widetilde{\Omega}_{\sigma \rho_0}}
    (k_2-k_1)^2 |w_1|^{ - \frac{2}{p-1}  }
    \leq
    \frac{4\overline{k}^2}{\rho_0^8} \int_{\widetilde{\Omega}_{\sigma \rho_0}}
    |w_1|^{ - \frac{2}{p-1}  } \leq \frac{4\overline{k}^2}{\rho_0^8} \sum_{j=1}^J\int_{B_r(x_j)}
    |w_1|^{ - \frac{2}{p-1}  }.
\end{equation}
By applying the $A_p$-property \eqref{eq-9-2} and the Lipschitz
propagation of smallness property \eqref{eq-9-1} to $w=w_1$ in $U
= \Omega \setminus B_{\overline{\sigma}\rho_0}(P_0)$,
with $\tau = \frac{r}{\rho_0}= \frac{\sigma}{2\max\{ 2,c_1,c_2
\}}$, and noticing that, for every $j$, $j=1,...,J$, $dist(x_j,
\partial U) \geq c_i r$, $i=1,2$, we have
\begin{equation}
    \label{eq-14-3}
    \int_{B_r(x_j)}  |w_1|^{ - \frac{2}{p-1}  }
    \leq
    \frac
    { B^{ \frac{1}{p-1}} |B_r(x_j)|    }
    {
    \left (
    \frac{1}{ |B_r(x_j)|  }
    \int_{B_r(x_j)}  |w_1|^2
    \right )^{ \frac{1}{p-1}   }   }
    \leq
    \frac
    { B^{ \frac{1}{p-1} }|B_r(x_j)|    }
    {
    \left (
    \frac{c_\tau}{ |B_r(x_j)|  }
    \int_{\Omega \setminus B_{\overline{\sigma}\rho_0}(P_0)}  |w_1|^2
    \right )^{ \frac{1}{p-1}   }   },
\end{equation}
where $B>0$, $p>1$ and $c_\tau>0$ only depend on $h$, $M_0$, $M_1$,
$M_2$, $\xi_0$, $\overline{k}$, $\sigma$ and on the frequency
ratio  $\mathcal{F}= \frac{ \|w_1\|_{ H^{ \frac{1}{2}   }(  \Omega
\setminus B_{\overline{\sigma}\rho_0}(P_0) )   } }
 {  \|w_1\|_{ L^{2}(  \Omega
\setminus B_{\overline{\sigma}\rho_0}(P_0) )   }  }$.
Such a bound can be achieved as follows.
Notice that, since $\Omega \setminus
B_{\overline{\sigma}\rho_0}(P_0)  \supset
B_{2\overline{\sigma}\rho_0}(P_0)  \setminus
B_{\overline{\sigma}\rho_0}(P_0) $, by applying \eqref{eq-8a-1},
we have
\begin{equation}
    \label{eq-15-1}
    \mathcal{F} \leq \frac{C}{\overline{\sigma} d},
\end{equation}
where $C>0$ only depends on $h$, $M_0$, $M_1$, $\xi_0$, $\xi_1$, $\overline{k}$.
By applying \eqref{eq-8a-1} and \eqref{eq-8b-3} to estimate {}from
below the denominator in the right hand side of \eqref{eq-14-3}, by
\eqref{eq-2-1} and \eqref{eq-14-2}, we obtain
\begin{equation}
    \label{eq-15-2}
    \int_{\widetilde{\Omega}_{\sigma \rho_0}}
    (k_2-k_1)^2 |w_1|^{ - \frac{2}{p-1}  }
    \leq
    \frac{C|\Omega|}{\rho_0^8(d^4f^2)^{\frac{1}{p-1}}} \leq \frac{C}{\rho_0^6(d^4f^2)^{\frac{1}{p-1}}},
\end{equation}
where $C>0$ only depends on  $h$, $M_0$, $M_1$, $M_2$, $\xi_0$, $\overline{k}$, $d$ and $\sigma$. By \eqref{eq-13-2} and \eqref{eq-15-2} we have
\begin{equation}
    \label{eq-16-1}
    \int_{\widetilde{\Omega}_{\sigma \rho_0}}
    (k_2-k_1)^2
    \leq
    \frac{C}{\rho_0^6 d^{\frac{4}{p}}} \epsilon^{ \frac{2s}{p(4+s)}  },
\end{equation}
where $C>0$ only depends on  $h$, $M_0$, $M_1$, $M_2$, $M_3$,
$\xi_0$, $\overline{k}$, $s$, $d$ and $\sigma$. Finally, by
\eqref{eq-12-2} and \eqref{eq-16-1}, the thesis follows.
\end{proof}

\bibliographystyle{plain}

\end{document}